\documentclass[11pt]{amsart}
\usepackage[colorlinks=true,citecolor=blue,linkcolor=blue]{hyperref}
\usepackage{amsfonts}
\usepackage[portuges,english]{babel}
\usepackage{graphicx}
\usepackage{amscd,color}
\usepackage{amsmath}
\usepackage{amssymb}
\usepackage{empheq,amsmath}
\theoremstyle{plain}

\newtheorem{corollary}{\bf Corollary}
\newtheorem{definition}{\bf Definition}
\newtheorem{example}{\bf Example}
\newtheorem{lemma}{\bf Lemma}

\newtheorem{proposition}{\bf Proposition}

\newtheorem{theorem}{\bf Theorem}
\numberwithin{equation}{section}

\newcommand{\real}{\mathbb{R}}

\newcommand{\Vol}{{\rm{Vol}}}
\newcommand{\di}{{\rm{ div}}}

\newcommand{\Ric}{{\rm{Ric}}}

\newcommand{\tr}{{\rm tr}\,}
\newcommand{\ric}{{\rm Ric}}
\newcommand{\R}{\mathcal{R}_f}

\makeatletter
\@namedef{subjclassname@2020}{%
  \textup{2020} Mathematics Subject Classification}
\makeatother

\title[Perelman singular manifolds ]{Perelman singular manifolds}

\author[M. Batista, A. Freitas, M. Santos]{M. Batista$^{1, \ast}$,  A. Freitas$^2$, and M.  Santos$^{2}$}

\address{$^1$CPMAT - IM\\
	Universidade Fe\-deral de Alagoas\\
	57072-900 Macei\'o, Alagoas, Brazil}
	\email{mhbs@mat.ufal.br}

\address{$^2$Departamento de Matem\'{a}tica\\
	Universidade Federal da Para\'{\i}ba\\
	58.051-900 Jo\~{a}o Pessoa, Para\'{\i}ba, Brazil}
\email{allan@mat.ufpb.br}
\email{marcio.academico@ufpb.br}


\thanks{$^\ast$Corresponding author.}

\begin{document} 

\begin{abstract}
On a Riemannian manifold with a smooth function $f: M\to \mathbb{R}$, we consider the linearization of the Perelman scalar curvature $\R$ and its $L^2$-formal adjoint operator $\delta\R^*$. A manifold endowed with a metric $g$ whose operator $\delta\R^*$ has a nontrivial kernel is called a Perelman singular manifold. In this paper, we present examples and apply general maximum principles to obtain rigidity or nonexistence results in the underlying setting.

\end{abstract}

\maketitle

\section{Introduction}

In the last few years, the study of smooth metric measure spaces has developed because it arises naturally in the study of self-shrinkers, Ricci solitons, harmonic heat flows, and many other settings. Its natural relevance in modern science can be seen, for example, because Ricci solitons play a crucial role in the theory of Ricci flow, and warped product Einstein metrics are of considerable interest in General Relativity. We highlight that the study of smooth metric measure spaces and their generalized curvatures traces back to Lichnerowicz \cite{l1, l2} and, more recently, Bakry and \'Emery \cite{BE} in the setting of diffusion processes, and it has been an active subject in recent years. For an overview, see, for instance, \cite{CGY},\cite{ grigoryan} and \cite{m}.

Throughout the paper, we use some notions that we introduce now. Let $\Sigma$ be an $n$-dimensional differentiable manifold with a possible non-empty boundary. We recall that a {\it smooth metric measure space} is a Riemanniana manifold $({\Sigma^n}, g)$ endowed with a real-valued smooth function $f: \Sigma \to \mathbb{R}$  which is used as density in the following way: $d\Vol_f = e^{-f}d\Vol$, where  $\Vol$ is the Riemannian measure of $\Sigma$, and $d\sigma_f=e^{-f}d\sigma$, where  $d\sigma$ is the Riemannian measure of $\partial\Sigma$, if the boundary is not empty. Hereafter, we denote a smooth metric measure space by $({\Sigma^n}, g, d\Vol_f )$.

Associated to this structure we have a second order differential operator defined by
$$\Delta_f u= e^f{\rm div}(e^{-f}\nabla u),$$
acting on space of smooth functions. This operator is known in the literature as {\it Drift Laplacian}.
Following  \cite{BE}, the natural generalization of Ricci curvature is defined by
$$\Ric_f=\Ric + \nabla^2 f,$$
which is known as {\it Bakry-\'Emery Ricci tensor} or simply by {\it weighted Ricci curvature}. We also recall the Perelman scalar curvature on $(\Sigma, g, d\Vol_f)$ which is defined by $$\mathcal{R}_f=R+2\Delta f-|\nabla f|^2,$$ where $R$ is the scalar curvature of $M,$ and was introduced by Perelman in \cite{gp}. Note that the Perelman scalar curvature is not, in general, the trace of the weighted Ricci curvature and so many differences in the studied theory arise.

Coming back to the classical case of the scalar curvature functional, the elements in the kernel of the adjoint of the underlying linearization, commonly referred to as static metrics, have inspired numerous developments in both mathematics and physics. We recall that a Riemannian metric $g$ on a manifold $M$ is called static if the linearized scalar curvature map has a nontrivial cokernel. This is equivalent to stating that the equation
\begin{align}
-(\Delta u)g+\nabla^2u-u\ric=0 \label{Staticequation}
\end{align}
admits some nontrivial solution. Such a solution is referred to as a \textit{static potential}. From the physical pointview, for example, a solution of (\ref{Staticequation}) in a 3-manifold allows to construct a space-time satisfying the vacuum Einstein equations (with cosmological constant), whose properties, phy\-si\-ca\-lly interpreted,
justify the name \textit{static} (see, for example, \cite{corvino}). Also, many interesting geometric properties of manifolds that admit a static potential have already been established. For example, the scalar curvature of such a manifold must be locally constant, $\Sigma=u^{-1}(0)$ (called apparent horizon) is a totally geodesic regular hypersurface, and the restrictions of $|\nabla u|$ to the connected components of the boundary (called surface gravities) are constants. For a good historical survey and demonstration of the some of the above mentioned properties, see \cite{HMR} and references therein.

This work focuses on studying properties and characterizations of Perelman singular manifolds, which are metrics in the kernel of the linearization of the Perelman scalar curvature. As recently observed by \cite{Ho} this study has a potential value to study the prescription of weighted scalar curvature on smooth metric measure spaces. Building on concepts and results from static metrics, we extend the theory to a weighted version. For instance, we observe that Perelman singular manifolds satisfy an interesting weighted version of the static equation:
\begin{equation}
-(\Delta_fu)\, g+\nabla^2u-u\ric_f=0. \label{perelmansing}
\end{equation}
In this direction, we investigate properties, examples, existence, and rigidity of these models, providing results in both compact and noncompact cases, and establishing their relationship with Ricci solitons. For example, inspired by an important rigidity result related to the area of the boundary in a 3-dimensional static manifold (given by Boucher, Gibbons, and Horowitz \cite{BGH} and Shen \cite{s}, independently), we obtain an estimate for the area of the boundary in a Perelman singular manifold (see Theorem \ref{thm1}).  We mention that area estimates are useful in the literature involving problems related to static metrics, Ricci solitons, critical metrics of the volume functional, and quasi-Einstein manifolds (see, for example, \cite{freitas} and references therein).  In addition to these estimates, we note the relation between Perelman singular manifolds and almost Ricci solitons, with an interesting rigidity result for the hemisphere (see Theorems \ref{thm2}). Furthermore, working in the noncompact case and using a generalized maximum principle, we provide a new characterization of the Gaussian Ricci Soliton (see Theorem \ref{thm3}).

As a final point of interest, we highlight the close relationship between Perelman singular manifolds and the construction of gradient Ricci solitons realized by warped products. It is noteworthy that the base of every gradient Ricci almost soliton realized as a warped product exhibits a structure given by an equation related to \eqref{perelmansing} (see \cite{naza1} and \cite{pina}).

{\bf Outline of the paper}:  In Section \ref{sct2}, we introduce the variational formulae and subsequently study the linearization of the Perelman scalar curvature. We observe that the elements in the kernel of the adjoint of the linearization of the Perelman scalar curvature must satisfy the equation (\ref{perelmansing}). With this understanding, we generate non-trivial examples of such a class. In Section \ref{sct3}, we investigate the existence of compact Perelman singular manifolds with a non-empty boundary, providing, in particular, an estimate for the area of the boundary in this setting, along with some integral conditions that relate this class to almost Ricci solitons. Finally, in Section \ref{sct4}, we employ known maximum principles to derive results on rigidity and non-existence in the complete noncompact case.

\section{Preliminaries}\label{sct2}
Following the idea of Fischer and Marsden as presented in \cite{FM}, we consider the Perelman map $\mathcal{R}_f : \mathcal{M}\to\mathbb{R}$, which associates each metric $g\in\mathcal{M}$ with its Perelman scalar curvature $\mathcal{R}_f(g)$, where $f$ is a fixed smooth function. Here, $\mathcal{M}$ is the space of Riemannian metrics on $\Sigma$. Furthermore, the symbols $\nabla f$, $\nabla^2f$ and $\Delta f$ denote the gradient, Hessian and the Laplacian of $f$ with respect to the metric $g$, respectively. After a straightforward computation, we obtain the linearization of $\mathcal{R}_f$ at a metric $g$ and its $L_f^2$-formal adjoint (here, $L^2_f$ denotes square integrable functions with respect to the weighted measure). Therefore,
\begin{proposition}\label{itens} The first variation of the Laplacian, the norm square of the gradient, and the scalar curvature are:
\begin{itemize}
\item[(i)]$\delta_h\Delta f=-\langle\nabla^2f, h\rangle-\langle \nabla f, {\rm div} h\rangle+\frac{1}{2}\langle \nabla f,\nabla \tr h\rangle;$
\item [(ii)]$\delta_h|\nabla f|^2=-h(\nabla f,\nabla f);$
\item[(iii)]$\delta_hR=-\Delta ({\rm tr}\, h)+{\rm div}({\rm div} h)-\langle h, {\rm Ric}\rangle, $
\end{itemize}
where $\delta_{h}$ means the variation of the metric in the direction of $h$.
\end{proposition}
\begin{proof}
See \cite[Thm. $1.174$ and Prop. $1.184.$]{B}. 
\end{proof}

Throughout this paper we use the notion of $f$-divergence, defined by
$${\rm div}_fh=e^{f}{\rm div}(e^{-f}h)={\rm div} h-h(\nabla f,\cdot),$$
where $h\in\mathcal{S}_0^{k,2}$, and $\mathcal{S}_0^{k,2}$ is the space of $k$ times differentiable $(0,2)$-symmetric tensors. Using that and the previous proposition, we announce the first variation of $\mathcal{R}_f$. The following expression is well-known in the literature for the Riemannian case. Here, for the sake of completeness, we present a demonstration in the weighted case.

\begin{proposition} The linearization of the map $g\mapsto \mathcal{R}_f(g)$ is given by
$$\delta_h\R(g)=-\Delta_f(\tr h)-\langle h, \ric_f\rangle+{\rm div}_f({\rm div}_fh), \, \, \mbox{for} \, \, h\in\mathcal{S}_0^{k,2}. $$
Moreover, its $L_f^2$-formal adjoint is
$$(\delta \mathcal{R}_f(g))^*u=-(\Delta_fu)\, g+\nabla^2u-u\ric_f, $$
for  $u\in C^\infty$, and $u|_{\partial\Sigma}=0$ if $\partial\Sigma\neq\emptyset$.
\end{proposition}
\begin{proof} Using the items in Proposition \ref{itens}, after a direct computation we get:
\begin{eqnarray}\label{derivadaRf}
\delta_h\R(g) &=&-\Delta (\tr h)+{\rm div}({\rm div} h)-\langle h, \ric\rangle-2\langle h, \nabla^2 f\rangle\\ \nonumber
& &-2\langle \nabla f,{\rm div} h\rangle+\langle\nabla f,\nabla \tr h\rangle +h(\nabla f,\nabla f)\\ \nonumber
&=&-\Delta_f(\tr h)-\langle h,\ric_f\rangle+{\rm div}({\rm div} h)-\langle h,\nabla^2f\rangle\\ \nonumber
&&-2\langle\nabla f,{\rm div} h\rangle+h(\nabla f,\nabla f),
\end{eqnarray}
and
\begin{eqnarray}\label{divf}
{\rm div}({\rm div}_fh)&=&{\rm div}({\rm div} h)-{\rm div}(h(\nabla f,\cdot))\nonumber\\
&=&{\rm div}({\rm div} h)-\langle {\rm div} h,\nabla f\rangle-\langle h,\nabla^2f\rangle.
\end{eqnarray}
Moreover,
\begin{eqnarray}\label{divdivf}
{\rm div}_f({\rm div}_fh)&=&{\rm div}({\rm div}_f h)-\langle {\rm div}_f h,\nabla f\rangle\nonumber\\
&=&{\rm div}({\rm div} h)-\langle {\rm div} h,\nabla f\rangle-\langle h,\nabla^2f\rangle-\langle {\rm div} h,\nabla f\rangle+h(\nabla f,\nabla f).
\end{eqnarray}

From equations \eqref{derivadaRf}, \eqref{divf} and \eqref{divdivf} we obtain
$$\delta_h\R=-\Delta_f(\tr h)-\langle h, \ric_f\rangle+{\rm div}_f({\rm div}_fh).$$
Using Stokes' theorem we obtain
$$\langle (\delta\R(g))^\ast u,h\rangle_{L_f^2}=\langle u, \delta_h\R(g)\rangle_{L_f^2}=\int_\Sigma\langle-\Delta_fu\, g+\nabla^2u-u \ric_f, h\rangle\, d\Vol_f,$$
for all $h\in \mathcal{S}_0^{k,2}$. So, the $L^2_f$-formal adjoint is given by 
$$ (\delta\mathcal{R}_f(g))^*u=-(\Delta_fu)\, g+\nabla^2u-u\ric_f,$$
and so we conclude the result.
\end{proof}

In \cite{CGY2}, the authors introduced the notion of $Q$-singular space based on the non-triviality of the kernel of $(\delta Q)^\ast$, where $Q$ is a curvature function associated with a geometric problem. It is worth noting that $Q$-singular spaces are an attractive subject, and numerous studies have been conducted; see \cite{CLY:19} for a general setting.

Building on their work, we introduce the following definition:

\begin{definition}
A complete smooth measure space $(\Sigma, g, d\Vol_f)$ is said a Perelman singular space $($or P-singular$)$  if $${\rm Ker}(\delta\mathcal{R}_f(g))^*\neq \{0\}.$$ 
\end{definition}
Sometimes we will write $(\Sigma, g, d\Vol_f, u)$ as a $P$-singular space meaning that $u$ is a non-trivial function in the kernel of $(\delta\mathcal{R}_f(g))^*$.
\begin{example}
Static manifolds are examples of $P$-singular spaces when $f$ is a constant function.
\end{example}

\begin{example}
A compact gradient steady Ricci soliton without boundary is a P-singular space. Indeed, since $\ric_f=0$, the function $u$ equal to one is in the kernel of $(\delta\mathcal{R}_f)^*$.
\end{example}

\begin{example}\label{Gaussian}  Consider the Euclidean space $\mathbb{R}^n$ equipped with its canonical metric $g_{\text{can}}$ and the Gaussian measure $\exp\left(-\frac{|x|^2}{2}\right)dx$. The triple, denoted as $(\mathbb{R}^n, g_{\text{can}}, \exp\left(-\frac{|x|^2}{2}\right)dx)$, is known as a Gaussian space. After a straightforward computation, we obtain:
$$(\delta\mathcal{R}_f)^*u(x) =-(\Delta u - \langle \nabla u, x\rangle)g_{can} +\nabla^2 u - ug_{can}.$$

For each $v\in\real^n$, consider $u_v : \real^n\to \real$ defined by $u_v(x)=\langle x, v\rangle.$ A direct computation give us that $\nabla u=v$ and $\nabla^2 u=0$. Moreover, $\ric_f = g_{can}$. So,
$$(\delta\mathcal{R}_f)^*u_v(x) = -(0-\langle v,x\rangle)g_{can}+ 0 -ug_{can} = ug_{can} - ug_{can}=0.$$
Thus, the Gaussian space is an example of a complete non-compact  P-singular space.
\end{example}

\begin{example}\label{sphere}Let $\mathbb{S}^n$ denote the unit sphere equipped with the canonical metric inherited from the Euclidean space. Let $v$ and $w$ be fixed vectors in $\mathbb{R}^{n+1}$ such that $\langle v, w\rangle=0$. Consider the measure $d\Vol_f = \exp(-\langle x, v\rangle), d\omega$, where $d\omega$ represents the Riemannian volume element on the sphere. Define the smooth function $u_w:\mathbb{S}^n\to \mathbb{R}$ as $u_w(x) = \langle x,w\rangle$.

A direct computation shows that $\nabla u_w = w^\top$ and $\nabla^2u_w = -u_wg_{can}$. Since  $\langle v, w\rangle=0$, after a straightforward computation we have that $$\Delta_fu_w = (-n+f)u_w \, \, \mbox{and} \, \, \ric_f = (n-1-f)g_{can}.$$
So,
$$(\delta\mathcal{R}_f)^*u_w(x) = (n-f)u_w\delta -u_wg_{can} - u_w(n-1-f)g_{can}=0,$$
and thus $(\mathbb{S}^n, \delta,  \exp(-\langle x, v\rangle)\, d\omega)$ is an example of a closed P-singular space.
\end{example}

Next, we will present some formulas that will help us characterize $P$-singular spaces. Consider a $P$-singular space denoted by $(\Sigma, g, e^{-f}d\text{Vol}, u)$. Through a straightforward computation, we derive the following expressions:

\begin{proposition}\label{for1} For $u\in C^\infty(\Sigma)$, the following holds:
\begin{enumerate}
\item[(i)] $\di_f(\ric_f) = \frac{1}{2}d\mathcal{R}_f $;
\item[(ii)] $\di_f((\Delta_fu)g) = d(\Delta_fu) - \Delta_fu\, df;$
\item[(iii)] $\di_f(\nabla^2u) = \di_f((\Delta_fu)g) + \ric_f(\nabla u,\cdot) + \Delta_fu\, df.$ 
\end{enumerate}
\end{proposition}
\begin{proof}
The first formula follows from $\di(\ric)=\frac{1}{2}dR$  and $\di(\nabla^2f) = d\Delta f + \ric(\nabla f,\cdot)$, which are consequences of the Bianchi and Ricci identities. The second and third formulas simply represent manipulations involving the $f$-divergence.
\end{proof}

\begin{corollary}\label{Rf} For $u \in \text{Ker}(\delta\mathcal{R}_f)^*$, the following holds:
$$\frac{1}{2}u\, d\mathcal{R}_f = \Delta_fu\, df.$$
\end{corollary}
\begin{proof} Taking the $f$-divergence of $(\delta\mathcal{R}_f(g))^*u$ and applying Proposition \ref{for1}, we get:
\begin{eqnarray*}
0&=& \di_f(-(\Delta_fu)g+\nabla^2u) - \di_f(u\ric_f)\\
&=& \ric_f(\nabla u,\cdot) + \Delta_fu\, df 
- u\di_f(\ric_f) - \ric_f(\nabla u,\cdot)\\
&=& \Delta_fu\, df - u\frac{1}{2}d\mathcal{R}_f,
\end{eqnarray*}
and so we conclude the result.

\end{proof}

\begin{corollary}\label{cor2.1} If $u\in{\rm Ker}(\delta\mathcal{R}_f)^*$, then there exists a smooth real function $\sigma$ on $\Sigma$ such that 
$$d\mathcal{R}_f = -2\sigma\cdot df\, \, \mbox{and} \, \, \Delta_fu = -\sigma\cdot u.$$
\end{corollary}
\begin{proof}
Fixing a point $p\in\Sigma$ and a vector $v\in T_p\Sigma$, we set $a=((d\mathcal{R}_f)_p(v), \Delta_fu )$ and $b=(u, -2df_p(v))$ as vectors in $\real^2$. By Corollary \ref{Rf}, we have $\langle a, b\rangle = 0$, implying that the vector $a$ is parallel to $b^\perp = (2df_p(v), u)$. The result follows from this property.\end{proof}

\begin{proposition}
Assume that $u\in{\rm Ker}(\delta\mathcal{R}_f)^*$ is a positive function. Then,
$$\Delta_{-\ln u}e^{-f} = -e^{-f}(\R - (n-1)\sigma),$$
where $\sigma$ is given in Corollary \ref{cor2.1}.
\end{proposition}
\begin{proof}
Tracing $(\delta\mathcal{R}_f(g))^*u$ and managing the expression, we obtain:
$$-(n-1)\Delta_fu - u\R = ue^f(\Delta e^{-f} + \langle\nabla\ln u, \nabla e^{-f}\rangle).$$
Using Corollary \ref{cor2.1} the result follows.
\end{proof}

By the end of this section, under suitable hypotheses on the Perelman scalar curvature, we are able to provide some information about the non-existence of $P$-singular manifolds in the closed case.
\begin{proposition}
Let $f$ be a smooth function on $M$ such that its set of critical points has measure zero. If $M$ is closed and the $P$-scalar curvature is constant, then $(M, g, d\Vol_f)$ cannot be $P$-singular.
\end{proposition}
\begin{proof}
Under our hypotheses, assume by contradiction that there exists a non-trivial function $u$ in the kernel of $(\delta\mathcal{R}_f(g))^*$. Using the fact that $\mathcal{R}_f$ is constant and the hypotheses on $f$, we deduce from Corollary \ref{Rf} that $\Delta_fu=0$. Applying the maximum principle, we conclude that $u$ is constant. Since $u$ belongs to the kernel of $(\delta\mathcal{R}_f)^*$, we obtain that $\ric_f=0$. Tracing $\ric_f$, we get $\mathcal{R}_f-\Delta_f f = 0.$ Integrating this equation and applying the divergence theorem, we conclude that $\mathcal{R}_f$ vanishes on $\Sigma$. Subsequently, by applying the maximum principle once again, we obtain that $f$ is constant, leading to a contradiction.
 \end{proof}
\medskip

Before introducing our next result, we recall that a Riemannian manifold $(\Sigma, g)$ endowed with a smooth function $f$ is called an expander Ricci soliton if there exists a negative constant $\rho$ such that $\Ric_f=\rho g$.

\begin{proposition}
Let $\Sigma$ be a Riemannian manifold satisfying$$\ric_f = (\lambda_0 +\lambda_1f)g \, \,  \mbox{and} \, \,  \mathcal{R}_f = c_0 +c_1f.$$ If $\Sigma$ is closed and $n\lambda_1-c_1<0$, then $f$ is constant and $\Sigma$ is Einstein. In particular, there is no closed expander Ricci soliton.
\end{proposition}
\begin{proof}
Tracing $\ric_f$, using the hypotheses and handling the expressions we get 
$$\Delta_ff + (n\lambda_1-c_1)f + (n\lambda_0-c_0)=0.$$ Analyzing the previous equality at the maximum and  the minimum points of $f$, we conclude that $f$ must be constant. Turning to the second part, for $\Ric_f =\lambda_0g$ and $\lambda_0<0$, we note that 
$$ \frac{1}{2}d\mathcal{R}_f={\rm{div}}_f(\Ric_f) ={\rm{div}}_f(\lambda_0 g) = -\lambda_0 df,$$
and so $\R= c_0 - 2\lambda_0f$, which verify the first part  of the lemma and so we conclude the result.
\end{proof}

\section{Rigidity and Nonexistence results for compact with boundary $P$-singular manifolds}\label{sct3}
In this section, under suitable hypotheses, we derive geometric conclusions for the case of a non-empty boundary. Throughout, we assume that $f$ is non-constant and $\Sigma$ is a compact manifold with a non-empty boundary $\partial\Sigma=u^{-1}(0)$, where $u>0$ in the interior of $\Sigma$.

\begin{proposition}
Let $(\Sigma, g, d\Vol_{f}, u)$ be a P-singular manifold with non-empty boundary $\partial\Sigma=u^{-1}(0)$, where $u>0$ in the interior of $\Sigma$. Then,
\begin{itemize}
\item [(i)] $0$ is a regular value of $u$;
\item[(ii)] $\partial\Sigma$ is totally geodesic;
\item[(iii)] $|\nabla u|$ is constant along each connected components of $\partial\Sigma$. 
\end{itemize}
\end{proposition}

\begin{proof}
The proof follows the steps outlined in the demonstrations of Propositions 2.3 and 2.6 in \cite{corvino}.
\end{proof}

Next, we establish a relationship between the boundary area and an integral involving the Perelman scalar curvature and the potential $u$.

\begin{proposition}
Let $(\Sigma, g, d\Vol_{f}, u)$ be a compact P-singular manifold such that $u>0$ on the interior of $\Sigma$, and $\partial\Sigma=u^{-1}(0)=\cup_{\alpha}\Gamma_\alpha$, where $\Gamma_{\alpha}$ are connected components. Then,
$$-(n-1)\sum_{\alpha}\kappa_{\alpha}\sigma_{f}(\Gamma_{\alpha})=\int_{\Sigma}\R u \, d\Vol_f,$$
where $\kappa_{\alpha}$ is the constant $|\nabla u|$  along the  component $\Gamma_{\alpha}$ of the boundary and  $\sigma_{f}(\Gamma_{\alpha})$ denotes the weighted area of the component $\Gamma_{\alpha}.$
\end{proposition}

\begin{proof}
Tracing $(\delta\mathcal{R}_f)^*u$ and making a direct computation we get 

$$\di_f\left(\nabla u - \frac{u}{n-1}\nabla f\right) +\frac{\mathcal{R}_f}{n-1}u=0.$$

Taking into account this, we have that
$$(n-1)\Delta_fu-{\rm{div}}_f(u\nabla f)=\R u.$$

Thus, applying the Stokes' theorem, we conclude that
\begin{equation}\label{inte}
(n-1)\int_{\partial\Sigma}\langle\nabla u,\nu\rangle\, d\sigma_f=\int_{\Sigma}\R u\, d\Vol_f.
\end{equation}

Follows from the previous proposition that $|\nabla u|$ is constant along of the boundary. Finally, plugging $\nu=-\frac{\nabla u}{|\nabla u|}$ into \eqref{inte} we conclude the desired result.
\end{proof}

With this in mind, we can obtain some non-existence results in this setting:

\begin{corollary}
There is no compact P-singular manifold $(\Sigma, g, d\Vol_{f}, u)$ such that $u>0$, on the interior of $\Sigma$, $\partial\Sigma=u^{-1}(0)$ and $\R\geq 0$.
\end{corollary}

\begin{proposition}
There is no $P$-singular manifold with non-empty boundary and constant $P$-scalar curvature.
\end{proposition}
\begin{proof}
Let $u$ be a function in the Kernel of $\delta\mathcal{R}_f^*$. Since $\mathcal{R}_f$ is constant, by Corollary \ref{Rf}, we have $\Delta_fu=0$. So, by the maximum principle,  $u$ vanishes.
\end{proof}

\begin{proposition}\label{no existence bdry}
There is no P-singular manifold with  non-empty boundary satisfying $\ric_f = \lambda g.$ 
\end{proposition}

\begin{proof}
We already know that $\Ric_f=\lambda g$ implies $\mathcal{R}_f = c_0-2\lambda f$. So, taking $u$ a smooth function in the kernel of $(\delta\mathcal{R}_f)^*$ and  applying Corollary \ref{Rf} we get $-\frac{u}{2}\cdot 2\lambda df = \Delta_fu df, $ which implies $\Delta_fu +\lambda u=0$. Plugging the function $u$ in the expression of $(\delta\mathcal{R}_f)^*$ and using the hypotheses we conclude that
$$\left\{ \begin{array}{lcl}
\nabla^2u=0& {\rm in} & \Sigma,    \\\medskip
u=0 & {\rm on} & \partial\Sigma,
\end{array}\right.$$
and thus we conclude that $u$ vanishes on $\Sigma$. 
\end{proof}

Now we introduce a useful result known in the literature as the Pohoz$\check{\mbox{a}}$ev-Sch$\ddot{\mbox{o}}$en identity which will play an important role in the next result: 
\begin{proposition} Let $\Sigma$ be an oriented compact smooth metric measure space and let $T$ be a $(0,2)$-symmetric tensor on $\Sigma$. Then, for all vector fields $X,$
\begin{equation}\label{PS}\int_{\partial\Sigma} T(X,\nu)\, d\sigma_f = \frac{1}{2}\int_\Sigma\langle T, \mathcal{L}_Xg\rangle\, d\Vol_f + \int_\Sigma (\di_f T)(X)\, d\Vol_f, \end{equation}
where $\nu$ is the outward unit normal vector field on $\partial\Sigma$.

\end{proposition}
\begin{proof} This formula follows immediately from  integration by parts.
\end{proof}

As an application of this formula we have the next result. We observe that the condition $-\Delta_{f}u=\lambda u$ with $\lambda$ constant, is equivalent to the condition that $\mathcal{R}_f=c_{0}-2\lambda f$ (see Corollary \ref{Rf}).

\begin{theorem}\label{thm1} Let $(\Sigma, g, d\Vol_f, u)$ be a P-singular space with $u$ positive on $\Sigma$ and $\partial\Sigma =\cup_{\alpha}\Gamma_\alpha $. Assuming that $\mathcal{R}_f = c_0+c_1f$ and $\partial_\nu f=0$ in $\partial\Sigma$, where $\nu$ is the outward unit normal vector field, then
$$(c_0+c_1)\sum_\alpha k_\alpha {\sigma}_f(\Gamma_\alpha) < \sum_{\alpha} k_\alpha \int_{\Gamma_\alpha}(\mathcal{R}_f^{\partial\Sigma} -c_1 f) \, d\sigma_f .$$ 

\end{theorem}
\begin{proof} Plugging $T=\ric_f+\frac{c_1}{2}g$ into the equation $\ref{PS}$ and using that $\di_f(\ric_f)=\frac{1}{2}d\mathcal{R}_f$, which follows from Proposition \ref{for1}, we get
$$\hspace{-0.45cm}\int_{\partial\Sigma}\ric_f(X, \nu)+\frac{c_1}{2}\langle X,\nu\rangle\, d\sigma_f= \frac{1}{2}\int_\Sigma\langle\ric_f+\frac{c_1}{2}g , \mathcal{L}_Xg\rangle\, d\Vol_f +\frac{1}{2}\int_\Sigma X(\mathcal{R}_f-c_1 f)\, d\Vol_f,$$
for all vector fields $X$ and $\mathcal{L}_Xg$ denotes the Lie derivative of $g$ with respect to $X$. 

Choosing $X=\nabla u$ and using that $\R-c_1 f$ is constant we have
$$\int_{\partial\Sigma}\ric_f(\nabla u, \nu)+\frac{c_1}{2}\langle \nabla u,\nu\rangle\, d\sigma_f= \frac{1}{2}\int_\Sigma\langle\ric_f+\frac{c_1}{2}g , 2\nabla^2u\rangle\, d\Vol_f.$$
Since $u$ belongs to ${\rm Ker}(\delta\mathcal{R}_f)^*$ and $\nu=-\frac{\nabla u}{|\nabla u|}$, from the hypothesis,  we have
$$-\int_{\partial\Sigma}|\nabla u|(\ric_f(\nu, \nu)+\frac{c_1}{2})\, d\sigma_f= \frac{1}{2}\int_\Sigma u|\ric_f+\frac{c_1}{2}g|^2 d\Vol_f.$$
Using the Gauss equation, after a straightforward computation, we get
$$\ric_f(\nu, \nu) = \frac{1}{2}\mathcal{R}_f - \frac{1}{2}\mathcal{R}_f^{\partial\Sigma} +\frac{1}{2}H_f^2 - \frac{1}{2}|A|^2,$$
where $\mathcal{R}_f^{\partial\Sigma}$ is the Perelman scalar curvature of the boundary. Since $\partial\Sigma$ is totally geodesic and $\partial_\nu f=0$ the equality reduce to $$ \ric_f(\nu, \nu) = \frac{1}{2}(\mathcal{R}_f - \mathcal{R}_f^{\partial\Sigma}),$$ and so
$$\int_{\partial\Sigma}|\nabla u|(\mathcal{R}_f - \mathcal{R}_f^{\partial\Sigma}+c_1)\, d\sigma_f\leq0.$$ Further, by Proposition \ref{no existence bdry}, the equality is not possible. To conclude, we observe that $|\nabla u|$ is constant on each connected component $\Gamma_\alpha$ and its value on it is denoted by $k_\alpha$. Using this and the value of the Perelman scalar curvature we have
$$\sum_\alpha k_\alpha\int_{\Gamma_\alpha} (c_0+c_1f - \mathcal{R}_f^{\partial\Sigma} +c_1)\, d\sigma_f < 0.$$
Manipulating this  inequality we get the desired result.
\end{proof}

Now,  we recall that the traceless tensor associated with a tensor $T$ is defined by

$$\overset\circ {T}=T-\frac{\tr(T)}{n}g.$$
With this notation, from equation \eqref{perelmansing} we can easily verify the following equation
\begin{equation}\label{traceless}
u\overset\circ {\Ric}_f=\overset\circ{\nabla^2}u.
\end{equation}
Next result provides sufficient conditions to  guarantee that a $P$-singular manifold is an almost Ricci soliton.

\begin{proposition}Let $(\Sigma, g, d\Vol_f, u)$ be a compact P-singular manifold with non-empty boundary.  Let $X= \frac{\R}{2}-\frac{e^{f}}{n}\nabla (e^{-f}(R+\Delta f))$ be a vector field along $\Sigma$ and assume that $\langle \nabla u, X\rangle\geq 0$.  If  $\overset\circ{\Ric_f}(\nabla u,\nabla u)\geq 0$ along of the boundary,  then $\Sigma$ is an almost Ricci soliton.

\end{proposition}
\begin{proof}
In fact, a straightforward calculation shows that
\begin{equation}\label{divergenceric}
{\rm{div}}_f\overset\circ{\Ric_f}=\frac{ d\R}{2}-\frac{e^{f}}{n} d(e^{-f}(R+\Delta f)).
\end{equation}

From $\eqref{PS},$ \eqref{divergenceric} and \eqref{traceless}   we get that 

$$-\int_{\partial\Sigma}|\nabla u|\overset\circ{\Ric}_f(\nu,\nu)\, d\sigma_f = \int_\Sigma u|\overset\circ{\Ric_f}|^{2}\, d\Vol_f + \int_\Sigma \langle X,\nabla u\rangle d\Vol_f,$$
where $\nu=-\frac{\nabla u}{|\nabla u|}.$
From the hypothesis on the Bakry-\'Emery Ricci tensor and the vector field $X$ we conclude the desired result.
\end{proof}
Our next finding describes how the hemisphere stands out as a unique P-singular manifold and almost Ricci soliton, given certain conditions. The statement of this result is as follows:

\begin{theorem}\label{thm2}
Let $(\Sigma, g, d\Vol_f, u)$ be a compact P-singular manifold with connected boundary.  If $\Sigma$ is an almost Ricci soliton  satisfying $\Ric_f=\omega g$ and $\sigma-\omega$ is a constant,  then $\Sigma$ is isometric to a hemisphere $\mathbb{S}^{n}_{+}.$
\end{theorem}

\begin{proof}
In fact, from the hypotheses we get that $\nabla u$ is a conformal vector field and therefore
\begin{equation}\label{conformal}
n\nabla^{2}u=\Delta u g.
\end{equation}

On another hand, since $\Sigma$ is $P$-singular and almost Ricci soliton we conclude that 
\begin{equation}\label{conformal2}
\nabla^{2}u=u(\omega-\sigma)g.
\end{equation}
In particular, $\omega-\sigma$ is a negative constant.

From equations \eqref{conformal} and \eqref{conformal2} we have that 
$$\nabla^{2}\Delta u=(\omega-\sigma)\Delta u g.$$
Taking into account that $\Delta u=0$ along of the boundary, from \cite{R1}, we conclude the desired result.
\end{proof}

\section{Rigidity and Non-existence results of open complete non-compact $P$-singular manifolds}\label{sct4}

In this section, we present some results about non-compact $P$-singular manifolds. A useful tool that we utilize is the following result, which first appeared with Yau; for a proof, see \cite{C}.

\begin{lemma}\label{div}Let $\Sigma$ be a Riemannian manifold, and let $X$ be a vector field on $\Sigma$ whose norm is integrable. If ${\rm {div}} X \geq 0$, then ${\rm {div}} X=0$ on $\Sigma$.
\end{lemma}

\begin{proposition}
Let $\Sigma$ be a complete Riemannian manifold endowed with a density function $f$. Assume that $1\in Ker (\delta\mathcal{R}_f)^*$.  If $|\nabla_{\nabla f}\nabla f| \in \mathcal{L}^1_f(\Sigma)$, then $\Sigma$ is isometric to $R\times \mathcal{S}$, where $\mathcal{S}$ is Ricci-flat and  $f(t,x)=t$.
\end{proposition}
\begin{proof}
Since $1\in Ker(\delta\mathcal{R}_f)^*$ we have $\Ric_f=0$ and $\mathcal{R}_f$ is constant. Thus $\Delta f = |\nabla f|^2 + c$, where $c$ is a constant. Plugging $f$ into the Bochner formula below
$$\frac{1}{2}\Delta_f |\nabla v|^2 = |\nabla^2 v|^2 +\langle \nabla v,\nabla\Delta_f v\rangle + \Ric_f(\nabla v, \nabla v),$$ 
which holds for any $v\in C^3$, and so we obtain
$$\frac{1}{2}\Delta_f |\nabla f|^2 = |\nabla^2 f|^2,$$
because $\Delta_ff$ is constant and $\Ric_f$ vanishes. Since $e^{-f}|\nabla|\nabla f|^2|$ is integrable on $\Sigma$ we are able to apply Lemma \ref{div} to get that $\nabla^2f$ vanishes. To conclude the result we use \cite[Theorem 2]{tashiro}.\end{proof}

\begin{proposition}
Let $\Sigma$ be a complete Riemannian manifold endowed with a density function $f$ such that $|\nabla f|$ is bounded. Assume that $1\in Ker (\delta\mathcal{R}_f)^*$ and its $P$-scalar curvature is a nonnegative. Then, $f$ is constant and $\Sigma$ is Ricci-flat.
\end{proposition}
\begin{proof} Under our hypotheses we already know that
$$\frac{1}{2}\Delta_f |\nabla f|^2 = |\nabla^2 f|^2.$$
By the Cauchy-Schwarz inequality we have:
$$\frac{1}{2}\Delta_f |\nabla f|^2 \geq \dfrac{(\Delta f)^2}{n} = \dfrac{(|\nabla f|^2 + \R)^2}{n}.$$
Since $\ric_f =0$, we are able to apply the weak Omori-Yau generalized maximum principle, see \cite{AMR:16}, and get that $\sup_\Sigma|\nabla f|^2 =\inf\R = 0$. Therefore $f$ is constant.

\end{proof}

The next result provides  a sufficient condition to  guarantee  that a complete non-compact $P$-singular manifold is a Ricci soliton.

\begin{proposition}
Let $(\Sigma, g, d\Vol_{f}, u)$ be a P-singular complete non-compact manifold  with $\Delta_fu=-\lambda u$, where $\lambda$ is a constant,  and  $\left \{p\in\Sigma; u(p)=0\right\}$ is a set of isolated points.  Suppose that $n\omega=R+\Delta f$ is a constant and $(\omega-\lambda)\langle\nabla u, \nabla f\rangle\geq 0.$  If $|\overset\circ {\Ric_f}(\nabla u)|\in\mathcal{L}^{1}_f(\Sigma)$, then $\Sigma$ is a gradient Ricci soliton.
\end{proposition}

\begin{proof}
Since $\Delta_fu=-\lambda u$ and $n\omega=R+\Delta f$ we have, from  \eqref{divergenceric}, that
\begin{equation}\label{divtraceless}
{\rm {div}}_f\overset\circ {\Ric_f}=\frac{d \mathcal{R}_f}{2}+\omega df.
\end{equation}

Since $\left \{p\in\Sigma; u(p)=0\right\}$ is a set of isolated points, from Corollary \ref{Rf}  and  $\Delta_fu=-\lambda u$ we conclude that
$$\frac{\nabla \mathcal{R}_f}{2}+\lambda \nabla f=0.$$

Plugging the previous equality into \eqref{divtraceless} we conclude that 

$$({\rm {div}}_f\overset\circ {\Ric_f})(\nabla u)=(\omega-\lambda)\langle\nabla f,\nabla u\rangle.$$

Thus,
\begin{equation}\label{divcomplete}
{\rm {div}}_f(\overset\circ {\Ric_f}(\nabla u))=u|\overset\circ {\Ric_f}|^{2}+ (\omega-\lambda)\langle\nabla u, \nabla f\rangle.
\end{equation}

Since  $|\overset\circ {\Ric_f}(\nabla u)|\in\mathcal{L}^{1}_f(\Sigma)$ we are able to use Lemma \ref{div}. Finally, from \eqref{divcomplete}, we conclude that $\stackrel\circ {\Ric_f}=0$ and, therefore, $\Sigma$ is a gradient Ricci soliton.

\end{proof}

\begin{theorem}\label{thm3}
Let $(\Sigma, g, d\Vol_{f}, u)$ be a non-trivial P-singular complete non-compact manifold with constant scalar curvature. Suppose that $\Ric_f=\omega g$ and $\Delta_fu=-\lambda u$, where $\lambda\neq 0$ is a constant, $\left \{p\in\Sigma; u(p)=0\right\}$ is a set of isolated points and $\omega-\lambda=\frac{R}{n}$.  Then, $\Sigma$ is a Gaussian Ricci soliton.
\end{theorem}

\begin{proof}
Indeed, since $Ric_f=\omega g$, we have that 

\begin{eqnarray*}
\R &=&R+\Delta f+\Delta f-|\nabla f|^{2}\\
&=&n\omega+(n\omega-R)-|\nabla f|^{2}\\
&=&2n\omega-R-|\nabla f|^{2}.
\end{eqnarray*}

Now,  since $\omega$ and the scalar curvature is constant we conclude that 
\begin{eqnarray*}
\nabla \R&=&-2\nabla^{2}f(\nabla f)\\
&=&-2(\omega\nabla f-\Ric(\nabla f)).
\end{eqnarray*}

Taking the divergence we get that:

\begin{equation}\label{laplacianRf}
\Delta (\R+2\omega f-\frac{2R}{n}f)=-2|\stackrel\circ {\Ric}|^2
\end{equation}

Since $\left \{p\in\Sigma; u(p)=0\right\}$ is a set of isolated points, we can invoke the Proposition \ref{Rf} to conclude that 
$$\R=-2\lambda f+c_0,$$
where $c_0$ is a constant.

Taking into account that $\lambda=\omega-\frac{R}{n}$ and from the above equality we guarantee that 
$$\Delta (\R+2\omega f-\frac{2R}{n}f)=0.$$

From the previous equality and \eqref{laplacianRf}, we conclude that $\Sigma$ is an Einstein manifold. Thus, applying Proposition 3.1 of \cite{PW}, we infer that either $\nabla^{2} f=0$ or $\Sigma$ is a Gaussian soliton. Now, assuming by contradiction that $\nabla^{2} f=0$, then $\R$ is constant, and consequently, $f$ is a constant. This leads to a contradiction, and we can thus conclude that $\Sigma$ is a Gaussian soliton.

\end{proof}


\section*{Funding}
The authors were partially supported by the Brazilian National Council for Scientific and Technological Development [Grants: 308440/2021-8, 405468/2021-0 to M.B., 316080/2021-7, 200261/2022-3 to A. F., 306524/2022-8 to M.S.], by Alagoas Research Foundation [Grant: E:60030.0000001758/2022 to M.B.], Para\'iba State Research Foundation (FAPESQ) [Grants: 2021/3175 to A.F and  3025/2021 to A.F and M.S.]  and authors were also partially supported by Coordination for the Improvement of Higher Education Personnel [Finance code - 001], Brazil.

\section*{Data availability statement}
This manuscript has no associated data.

\section*{Conflict of interest statement}
On behalf of all authors, the corresponding author states that there is no conflict of interest.

\bibliographystyle{plain}
\bibliography{reference.bib}

\begin{thebibliography}{10}

\bibitem{AMR:16}
Luis~J. Al{\'{\i}}as, Paolo Mastrolia, and Marco Rigoli.
\newblock {\em Maximum principles and geometric applications}.
\newblock Springer Monogr. Math. Cham: Springer, 2016.

\bibitem{BE}
Dominique Bakry and Michel {\'E}mery.
\newblock Diffusions hypercontractives.
\newblock S{\'e}min. de probabilit{\'e}s {XIX}, {Univ}. {Strasbourg} 1983/84,
  {Proc}., {Lect}. {Notes} {Math}. 1123, 177-206 (1985)., 1985.

\bibitem{B}
Arthur~L. Besse.
\newblock {\em Einstein manifolds}.
\newblock Class. Math. Berlin: Springer, reprint of the 1987 edition edition,
  2008.

\bibitem{BGH}
W.~Boucher, G.~W. Gibbons, and Gary~T. Horowitz.
\newblock Uniqueness theorem for anti-de sitter spacetime.
\newblock {\em Phys. Rev. D}, 30:2447--2451, Dec 1984.

\bibitem{C}
Antonio Caminha.
\newblock The geometry of closed conformal vector fields on {Riemannian}
  spaces.
\newblock {\em Bull. Braz. Math. Soc. (N.S.)}, 42(2):277--300, 2011.

\bibitem{CLY:19}
Jeffrey~S. Case, Yueh-Ju Lin, and Wei Yuan.
\newblock Conformally variational {Riemannian} invariants.
\newblock {\em Trans. Am. Math. Soc.}, 371(11):8217--8254, 2019.

\bibitem{CGY2}
Sun-Yung~A. Chang, Matthew~J. Gursky, and Paul Yang.
\newblock {\em Remarks on a fourth order invariant in conformal geometry}.
\newblock Asp. Math. HKU, proceedings of a conference in algebra, geometry and
  several complex variable edition, June 1996.

\bibitem{CGY}
Sun-Yung~A. Chang, Matthew~J. Gursky, and Paul Yang.
\newblock Conformal invariants associated to a measure.
\newblock {\em Proc. Natl. Acad. Sci. USA}, 103(8):2535--2540, 2006.

\bibitem{corvino}
Justin Corvino.
\newblock Scalar curvature deformation and a gluing construction for the
  {Einstein} constraint equations.
\newblock {\em Commun. Math. Phys.}, 214(1):137--189, 2000.

\bibitem{naza1}
F.~E.~S. Feitosa, A.~A.~Freita Filho, and J.~N.~V. Gomes.
\newblock On the construction of gradient {Ricci} soliton warped product.
\newblock {\em Nonlinear Anal., Theory Methods Appl., Ser. A, Theory Methods},
  161:30--43, 2017.

\bibitem{FM}
Arthur~E. Fischer and Jerrold~E. Marsden.
\newblock Deformations of the scalar curvature.
\newblock {\em Duke Math. J.}, 42:519--547, 1975.

\bibitem{freitas}
A.~Freitas and M.~Santos.
\newblock Boundary topology and rigidity results for generalized {{\(( \lambda
  , n + m)\)}}-{Einstein} manifolds.
\newblock {\em Ann. Mat. Pura Appl. (4)}, 199(6):2511--2520, 2020.

\bibitem{grigoryan}
Alexander Grigor'yan.
\newblock {\em Heat kernel and analysis on manifolds}, volume~47 of {\em AMS/IP
  Stud. Adv. Math.}
\newblock Providence, RI: American Mathematical Society (AMS); Somerville, MA:
  International Press, 2009.

\bibitem{HMR}
Oussama Hijazi, Sebasti{\'a}n Montiel, and Simon Raulot.
\newblock Uniqueness of the de {Sitter} spacetime among static vacua with
  positive cosmological constant.
\newblock {\em Ann. Global Anal. Geom.}, 47(2):167--178, 2015.

\bibitem{Ho}
Pak~Tung Ho and Jinwoo Shin.
\newblock Deformation of the weighted scalar curvature.
\newblock {\em SIGMA, Symmetry Integrability Geom. Methods Appl.}, 19:paper
  087, 15, 2023.

\bibitem{pina}
M{\'a}rcio Lemes~de Sousa and Romildo Pina.
\newblock Gradient {Ricci} solitons with structure of warped product.
\newblock {\em Result. Math.}, 71(3-4):825--840, 2017.

\bibitem{l1}
A.~Lichnerowicz.
\newblock Vari{\'e}t{\'e}s riemanniennes {\`a} tenseur {{\(C\)}} non
  n{\'e}gatif.
\newblock {\em C. R. Acad. Sci., Paris, S{\'e}r. A}, 271:650--653, 1970.

\bibitem{l2}
Andre Lichnerowicz.
\newblock Vari{\'e}t{\'e}s k{\"a}hl{\'e}riennes {\`a} premi{\`e}re classe de
  {Chern} non negative et vari{\'e}t{\'e}s riemanniennes {\`a} courbure de
  {Ricci} g{\'e}n{\'e}ralis{\'e}e non negative. ({K{\"a}hlerian} manifolds of
  non-negative first {Chern} class and {Riemannian} manifolds with non-negative
  generalized {Ricci} curvature).
\newblock {\em J. Differ. Geom.}, 6:47--94, 1971.

\bibitem{m}
Frank Morgan.
\newblock Manifolds with density.
\newblock {\em Notices Am. Math. Soc.}, 52(8):853--858, 2005.

\bibitem{gp}
Grisha Perelman.
\newblock The entropy formula for the {Ricci} flow and its geometric
  applications.
\newblock Preprint, {arXiv}:math/0211159 [math.{DG}] (2002)., 2002.
\newblock Id/No 0211159.

\bibitem{PW}
Peter Petersen and William Wylie.
\newblock Rigidity of gradient {Ricci} solitons.
\newblock {\em Pac. J. Math.}, 241(2):329--345, 2009.

\bibitem{R1}
Robert~C. Reilly.
\newblock Geometric applications of the solvability of {Neumann} problems on a
  {Riemannian} manifold.
\newblock {\em Arch. Ration. Mech. Anal.}, 75:23--29, 1980.

\bibitem{s}
Ying Shen.
\newblock A note on {Fischer}-{Marsden}'s conjecture.
\newblock {\em Proc. Am. Math. Soc.}, 125(3):901--905, 1997.

\bibitem{tashiro}
Yoshihiro Tashiro.
\newblock Complete {Riemannian} manifolds and some vector fields.
\newblock {\em Trans. Am. Math. Soc.}, 117:251--275, 1965.

\end{thebibliography}

\end{document}